\documentclass{svjour3}                     
\smartqed  
\usepackage{graphicx}
\usepackage{amssymb,amsfonts}
\usepackage{amsmath}
\usepackage{algorithm}
\usepackage{algorithmic}
\usepackage{latexsym}
\numberwithin{theorem}{section}
\numberwithin{definition}{section}
\numberwithin{proposition}{section}
\numberwithin{lemma}{section}
\numberwithin{remark}{section}
\numberwithin{example}{section}
\numberwithin{algorithm}{section}

\begin{document}

\title{Optimal Solution of Nonlinear Fuzzy Optimization Problem under Linear Order Relation}

\titlerunning{Optimal Solutions of Nonlinear Fuzzy Optimization Problems}        

\author{Umme Salma M. Pirzada }

\authorrunning{U. M. Pirzada} 

\institute{Umme Salma M. Pirzada \at
              School of Science and Engineering, Navrachana University, Vaodadara-391410, India. \\
	      \email{salmapirzada@yahoo.com}  
}

\date{}

\maketitle

\begin{abstract}
Multi-variable nonlinear fuzzy optimization problem is considered under linear order relation on fuzzy numbers. Using gH-differentiability of a fuzzy-valued function $\tilde{f}$, new necessary and sufficient optimality conditions are proposed. The optimality conditions are obtained without putting additional conditions on fuzzy-valued functions like, convexity, quasi-convexity, pseudo-convexity. Optimum solution of the fuzzy optimization problem is obtained based on the optimality conditions. Illustrations and a case study are given to explain the numerical applications of the proposed results.  Comparison of optimality conditions from existing literature is given.
\end{abstract}

\section{Introduction}
\label{intro}
\begin{par}
Fuzzy optimization accounts for any imprecision in the optimization problem. Bellman and Zadeh\cite{BE} introduced fuzzy 
optimization problems, they have stated that a fuzzy decision can be viewed as the intersection of fuzzy 
goals and problem constraints. There after a large number of solutions of linear and non-linear fuzzy optimization problems
are available in the literature. The optimality conditions for fuzzy optimization problems have been proposed by Wu\cite{HS4}\cite{HS5}
\& Pathak and Pirzada\cite{PA}. 
\end{par}
\begin{par}
Pathak and Pirzada\cite{PA}, have studied the necessary and sufficient optimality conditions to find the non-dominated 
solution of multi-variable fuzzy optimization problem. They have used the Hukuhara differentiability and partial order relation on 
fuzzy numbers.  Hukuhara differentiability is based on Hukuhara difference. Unfortunately, Hukuhara difference may not always 
exists, for instance, for fuzzy-valued function $\tilde{f}(x) = \tilde{a} \odot g(x)$, the conditions for existence of 
Hukuhara differentiability are :  the real-valued function $g$ must be defined on positive domain, that is,
 $g: (a,b) \to \mathbb{R}_{+}$, $(a,b) \subset \mathbb{R}_{+}$ with $g^{\prime}(x) > 0 $.  Moreover, if function $\tilde{f}$ 
is not H-differentiable then its level-functions are also not differentiable. 
\end{par}
\begin{par}
Stefanini\cite{ST1,ST2} generalized the Hukuhara difference, the generalized Hukuhara difference is usually termed as gH-difference.
Bede and Stefanini\cite{BE1} have defined Generalized Hukuhara differentiability ( gH-differentiability ) of 
fuzzy-valued functions based on gH-difference. The fuzzy-valued function $\tilde{f}: \mathbb{R} \to F(\mathbb{R})$ with 
$\tilde{f}(x) = \tilde{a} \odot g(x)$ where $g$ is a real-valued function with $g^{\prime}(x) < 0 $ is gH-differentiable 
though its lower and upper level functions are not differentiable. 
\end{par}
\begin{par}
The ranking of fuzzy numbers plays the important role in the study of fuzzy-optimization. Usually partial order relation
(fuzzy max order) is adopted to compare fuzzy numbers, unfortunately under fuzzy-max order relation, fuzzy numbers are not
always comparable. Hence we use alternative approach proposed by Goetschel and Voxman\cite{GO} called linear order relation on
on fuzzy numbers under which fuzzy numbers are comparable in the natural way.
\end{par}
\begin{par}
Here we consider unconstrained multi-variable fuzzy optimization problem with linear order relation defined on set of fuzzy numpers.
The concept of gH-differentiability of fuzzy-valued function over $\mathbb{R}$ have been used. The necessary and sufficient conditions
to find local and global minimizer of fuzzy-valued functions have been obtained and the optimal solution of non-linear fuzzy optimization
problems under a linear order relation on fuzzy numbers have been found.
\end{par}

\begin{par}
The paper is organized as follows: In section 2, we cite some basic definitions and results on fuzzy numbers, linear 
order relation on fuzzy numbers and gH-differentiability of a fuzzy-valued function over $\mathbb{R}$. The gH-differentiability 
of a fuzzy-valued function over $\mathbb{R}^{n}$ is defined in the same section. The local and global minimizer of an 
unconstrained multi-variable fuzzy optimization problem based on the optimality conditions have been found in Section 3. 
In Section 4, appropriate examples and case study study are given to illustrate the application of proposed results. Moreover, we correct 
the Example 1 proposed by Pathak and Pirzada\cite{PA}. The section 5 contains conclusion.
\end{par}

\section{Preliminaries}
We start with definition of fuzzy numbers.
\label{sec:2}
\begin{definition}\label{def1}\cite{GE} Let $\mathbb{R}$ be the set of real numbers and $\tilde{a}:\mathbb{R} \to[0,1]$ be a fuzzy set. The fuzzy set $\tilde{a}$ is called a fuzzy number if it satisfies the following properties:
\begin{description}
        \item[(i)] {$\tilde{a}$ is normal, that is, there exists $r_0 \in \mathbb{R} $ such that $\tilde{a}(r_0)=1$;}
        \item[(ii)] {$\tilde{a}$ is fuzzy convex, that is,  $\tilde{a}(t r + (1-t) s)\geq \min \{\tilde{a}(r),\tilde{a}(s)\}$, whenever r, s $\in
\mathbb{R}~~ and~~~t \in [0,1]$;}
        \item[(iii)] {$\tilde{a}(r)$ is upper semi-continuous on $\mathbb{R}$, that is, $\{r / \tilde{a}(r) \geq \alpha \}$ is a closed subset of $\mathbb{R}$ for each $\alpha \in (0,1]$;}
        \item[(iv)] {$cl\{r \in \mathbb{R} / \tilde{a}(r) >0\}$ forms a compact set,}
    \end{description}

where $cl$ denotes closure of a set. The set of all fuzzy numbers on $\mathbb{R}$ is denoted by $F(\mathbb{R})$. 
\end{definition}
The $\alpha$-level set for a fuzzy number is defined as follows:
\begin{definition}
For all $\alpha \in (0,1]$, $\alpha$-level set $\tilde{a}_{\alpha}$ of any $\tilde{a}\in F(\mathbb{R})$ is defined as 

\begin{eqnarray*}\tilde{a}_{\alpha} = \{ r \in \mathbb{R}/ \tilde{a}(r)\geq \alpha \}. \end{eqnarray*} 

The 0-level set $\tilde{a}_{0}$ is defined as the closure of the set $\{r \in \mathbb{R} / \tilde{a}(r) >0\}$. 
\end{definition}
\begin{remark}
By definition of fuzzy numbers, we can prove that, for any $\tilde{a}\in F(\mathbb{R})$ and for each $\alpha \in (0,1]$ , $\tilde{a}_{\alpha}$ is compact convex subset of $\mathbb{R}$, and we write $\tilde{a}_{\alpha}=[{\tilde{a}_{\alpha}}^{L},{\tilde{a}_{\alpha}}^{U}]$.
\end{remark}
\begin{remark}
The fuzzy number $\tilde{a}\in F(\mathbb{R})$ can be recovered from its $\alpha$-level sets by a well-known decomposition theorem (ref. \cite{GE1}), which states that $\tilde{a}= \cup_{\alpha \in [0,1]} \alpha \cdot \tilde{a}_{\alpha} $ where union on the right-hand side is the standard fuzzy union.
\end{remark}

\begin{definition}\label{def2}\cite{SO}  Addition and scalar multiplication in the set of fuzzy numbers $F(\mathbb{R})$ by their $\alpha$-level sets, are defined as follows:
\begin{eqnarray*}
(\tilde{a}\oplus \tilde{b})_{\alpha} & = &
[\tilde{a}_{\alpha}^{L}+\tilde{b}_{\alpha}^{L}, \tilde{a}_{\alpha}^{U}+\tilde{b}_{\alpha}^{U}] \\
(\lambda \odot \tilde{a})_{\alpha} & = &
[\lambda\cdot\tilde{a}_{\alpha}^{L},\lambda\cdot\tilde{a}_{\alpha}^{U}],~if~\lambda \geq 0 \\
			           & = &
[\lambda\cdot\tilde{a}_{\alpha}^{U},\lambda\cdot\tilde{a}_{\alpha}^{L}],~if~\lambda < 0
\end{eqnarray*}
where $\tilde{a}, \tilde{b} \in F(\mathbb{R})$, $\lambda \in \mathbb{R}$ and $\alpha \in [0,1]$.
\end{definition}

\begin{definition} \cite{SH} \label{def3}
A triangular fuzzy number $\tilde{a}$ is defined using three real numbers $a^{L},a$ and $a^{U}$ as
\[
{\tilde{a}}(r) = \left\{
\begin{array}{ll}
 {\frac {(r-a^{L})}{(a- a^{L})}}~~~ if ~~ a^{L} \leq r \leq a  \\
 {\frac {(a^{U}-r)}{(a^{U}-a)}}~~~ if ~~ a < r \leq a^{U} \\
 0~~~~~~~~~~~~~otherwise
\end{array} \right.
\]
and denoted by $\tilde{a} = (a^{L},a,a^{U})$. The $\alpha$-level set of $\tilde{a}$ is then 
\[
\tilde{a}_{\alpha} = [(1-\alpha)a^{L}+\alpha a, (1-\alpha)a^{U}+\alpha a], 
\] for $\alpha \in [0,1]$.
\end{definition} 

Goetschel and Voxman \cite{GO} defined an ordering, $\preceq$ for all $\tilde{a} \in F(\mathbb{R})$ such that $\tilde{a}_{\alpha}^{L}$ and $\tilde{a}_{\alpha}^{U}$ are Lebesgue integrable with respect to $\alpha \in [0,1]$.

\begin{definition}\label{def4}\cite{GO}
Suppose that $\tilde{a}, \tilde{b} \in F(\mathbb{R})$ with their $\alpha$-level sets $\tilde{a}_{\alpha} = [\tilde{a}_{\alpha}^{L}, \tilde{a}_{\alpha}^{U}]$ and $\tilde{b}_{\alpha} = [\tilde{b}_{\alpha}^{L}, \tilde{b}_{\alpha}^{U}]$, for $\alpha \in [0,1]$ such that $\tilde{a}_{\alpha}^{L}$ , $\tilde{a}_{\alpha}^{U}$, $\tilde{b}_{\alpha}^{L}$, $\tilde{b}_{\alpha}^{U}$  are Lebesgue integrable with respect to $\alpha$. Then $\tilde{a} \preceq \tilde{b}$ if
\begin{eqnarray*}
\int_{0}^{1} {\alpha [ \tilde{a}_{\alpha}^{L}+\tilde{a}_{\alpha}^{U}]}d\alpha \leq \int_{0}^{1} {\alpha [ \tilde{b}_{\alpha}^{L}+\tilde{b}_{\alpha}^{U}]}d\alpha
\end{eqnarray*}
\end{definition}

\begin{remark}
The ordering $\preceq$ is reflexive and transitive; moreover,any two elements of $F(\mathbb{R})$ are comparable under the ordering $\preceq$, i.e., $\preceq$ is a linear ordering for $F(\mathbb{R})$, (refer \cite{GO}).
\end{remark}

Based on the Definition \ref{def4}, we say that $\tilde{a} \prec \tilde{b}$ if

\begin{eqnarray*}
\int_{0}^{1} {\alpha [ \tilde{a}_{\alpha}^{L}+\tilde{a}_{\alpha}^{U}]}d\alpha < \int_{0}^{1} {\alpha [ \tilde{b}_{\alpha}^{L}+\tilde{b}_{\alpha}^{U}]}d\alpha
\end{eqnarray*}
Now we state the definition of Hausdorff metric on set of fuzzy numbers.

\begin{definition}\label{def5}\cite{PI}
Let $A, B \subseteq \mathbb{R}^{n}$. The Hausdorff metric $d_H $ is defined by
\begin{eqnarray*}
d_H(A,B) \mathrel{\mathop:}= \max\{\sup_{x \in A}\inf_{y \in B}\|x-y\|, \sup_{y \in B}\inf_{x \in A}\|x-y\|\}. \end{eqnarray*}
Then the metric $d_{F}$ on $F(\mathbb{R})$ is defined as 
\begin{eqnarray*}
d_{F}(\tilde{a}, \tilde{b})  \mathrel{\mathop:}=  \sup_{0 \leq \alpha \leq 1}\{d_H({\tilde{a}}_{\alpha},{\tilde{b}}_{\alpha})\},
\end{eqnarray*}
for all $\tilde{a}, \tilde{b} \in F(\mathbb{R})$. Since $\tilde{a}_{\alpha}$ and $\tilde{b}_{\alpha}$ are compact intervals in $\mathbb{R}$,
\begin{eqnarray*}
d_{F}(\tilde{a}, \tilde{b})  =  \sup_{0 \leq \alpha \leq 1} \max\{
|{{\tilde{a}}_{\alpha}}^{L}-{{\tilde{b}}_{\alpha}}^{L}|,
|{{\tilde{a}}_{\alpha}}^{U}-{{\tilde{b}}_{\alpha}}^{U}|\}.
\end{eqnarray*}
\end{definition}

\begin{definition}\cite{HS1}  \label{def6}
Let $V$ be a real vector space and $F(\mathbb{R})$ be a set of fuzzy numbers. Then a function $\tilde{f}: V \to F(\mathbb{R})$ is called fuzzy-valued function defined on $V$. For each $x \in V$, $\tilde{f}(x)$ is a fuzzy number. Corresponding to this function $\tilde{f}$, for $\alpha \in [0,1]$, we have two real-valued functions $\tilde{f}_{\alpha}^L$ and $\tilde{f }_{\alpha}^U$ on $V$ as ${\tilde{f}_{\alpha}}^{L}(x) = (\tilde{f}(x))_{\alpha}^{L}$ and ${\tilde{f}_{\alpha}}^{U}(x) = (\tilde{f}(x))_{\alpha}^{U}$ for all $x \in V$. These functions $\tilde{f}_{\alpha}^L(x)$ and $\tilde{f }_{\alpha}^U(x)$ are called $\alpha$-level functions of the fuzzy-valued function $\tilde{f}$.
\end{definition}
The continuity of a fuzzy-valued function can be stated form \cite{DI} as 

\begin{definition}\label{def7}\cite{DI} Let $\tilde{f}: \mathbb{R}^{n}\to F(\mathbb{R}) $ be a fuzzy-valued
function. We say that $\tilde{f}$ is continuous at
$c \in \mathbb{R}^{n}$ if and only if for every $\epsilon >0$, there exists a
$\delta = \delta(c,\epsilon) >0$ such that 
\[d_{F}(\tilde{f}(x), \tilde{f}(c))< \epsilon\]
for all $x\in \mathbb{R}^{n}$ with $\|x-c\| < \delta$. That is,
\begin{eqnarray*}
 {\lim}_{x\rightarrow c}{\tilde{f}(x) \mathrel{\mathop:}= \tilde{f}(c)}.
\end{eqnarray*}
\end{definition}
\begin{example}
Let $\tilde{f}: X \to F(\mathbb{R}$, $X$ is an open subset of $\mathbb{R}_{+}$, be a fuzzy-valued function defined by $\tilde{f}(x) = \tilde{a} \odot x$, $x \in X$ and $\tilde{a}$ is a fuzzy number. The $\tilde{f}$ is a continuous fuzzy-valued function since, we find an $\epsilon >0$, for a $\delta > 0$ and $c \in X$ such that $|x-c| < \delta$ implies
\begin{eqnarray*}
d_{F}(\tilde{f}(x), \tilde{f}(c)) 
& = & \sup_{0 \leq \alpha \leq 1} \max\{ |{{\tilde{a}}_{\alpha}}^{L}x -{{\tilde{a}}_{\alpha}}^{L}c|, |{{\tilde{a}}_{\alpha}}^{U}x-{{\tilde{a}}_{\alpha}}^{U}c|\}  \\
& = & |x - c| \sup_{0 \leq \alpha \leq 1} \max\{ |{{\tilde{a}}_{\alpha}}^{L}|, |{{\tilde{a}}_{\alpha}}^{U}|\} \\
& < & \delta d_{F}(\tilde{a}, \tilde{0}) = \epsilon,
\end{eqnarray*}
where $\epsilon = \delta {{\tilde{a}}_{0}}^{U}$. Therefore, $\tilde{f}$ is continuous function on $X$ subset of $\mathbb{R}_{+}$.
\end{example}

Generalized Hukuhara difference (gH-difference) is defined as follows:
\begin{definition}\cite{ST1,ST2}\label{def8}
Given two fuzzy numbers $\tilde{a}, \tilde{b} \in F(\mathbb{R})$, the gH-difference is the fuzzy number $\tilde{c}$, if exists, such that 
\begin{equation}
\tilde{a} \ominus_{gH} \tilde{b} = \tilde{c} \Leftrightarrow \left\{
\begin{array}{ll}
 ~~~(i) ~\tilde{a} = \tilde{b} + \tilde{c}, \\
 or~ (ii)~ \tilde{b} = \tilde{a} - \tilde{c}. 
\end{array} \right.
\end{equation}
\end{definition}
In terms of $\alpha$-level sets, $[\tilde{a} \ominus_{gH} \tilde{b}]_{\alpha} = [\min\{ \tilde{a}_{\alpha}^{L} - \tilde{a}_{\alpha}^{L}, \tilde{a}_{\alpha}^{U} - \tilde{a}_{\alpha}^{U}\}, \max\{ \tilde{a}_{\alpha}^{L} - \tilde{a}_{\alpha}^{L}, \tilde{a}_{\alpha}^{U} - \tilde{a}_{\alpha}^{U} \}]$. 
\begin{par}
The generalized Hukuhara differentiability (gH-differentiability ) of the fuzzy-valued function based on gH-difference is defined as:
\end{par}
\begin{definition}\cite{BE1}\label{def9}
Let $x_{0} \in X$, $X$ is an open interval and $h$ be such that $x_{0} + h \in X$, then gH-derivative of a function $\tilde{f}: X \to F(\mathbb{R})$ at $x_{0}$ is defined as
\begin{equation}\label{eq}
\tilde{f}^{\prime}_{gH}(x_{0}) = \lim_{h \to 0} {{{1} \over {h}} \tilde{f}(x_{0} + h ) \ominus_{gH} \tilde{f}(x_{0}). }                                      
\end{equation}
If $\tilde{f}^{\prime}_{gH}(x_{0}) \in F(\mathbb{R})$ satisfying (\ref{eq}) exists, we say that $\tilde{f}$ is generalized Hukuhara differentiable ( gH-differentiable ) at $x_{0}$. If $\tilde{f}$ is gH-differentiable at each $x \in X $ with gH-derivative $\tilde{f}^{\prime}_{gH}(x) \in F(\mathbb{R})$, we say that $\tilde{f}$ is gH-differentiable on $X$. Moreover, we can define continuously gH-differentiable fuzzy-valued function using continuity of gH-differentiable fuzzy-valued function. 
\end{definition}
\begin{example}
The gH-differentiable fuzzy-valued functions are listed in \textbf{Table 2}.
\end{example}
The partial gH-derivative of a fuzzy-valued function $\tilde{f}$ on $\mathbb{R}^n$ can be defined in a same manner as the partial H-derivative of a fuzzy-valued function defined in \cite{PI}. 
\begin{definition} \label{def10}
Let $\tilde{f}$ be a fuzzy-valued function defined on an open subset $X$ of $\mathbb{R}^n$ and 
let $\bar{x}^{0} = (x^{0}_{1},...,x^{0}_{n}) \in X$ be fixed. \\ We say that $\tilde{f}$ has the 
$i^{th}$ partial gH-derivative $D_{i}\tilde{f}(\bar{x}^{0})$ at $\bar{x}^{0}$ if and only if the fuzzy-valued function 
$\tilde{g}(\bar{x}_{i}) =\tilde{f}(x^{0}_{1},..,x^{0}_{i-1},x_{i},x^{0}_{i+1},..,x^{0}_{n})$ is gH-differentiable at $x^{0}_{i}$ with gH-derivative $D_{i}\tilde{f}(\bar{x}^{0})$. We also write 
$D_{i}\tilde{f}(\bar{x}^{0})$ as $({\partial\tilde{f}}/ {\partial x_{i}})(\bar{x}^{0})$.
\end{definition}
We define gH-differentiability of the fuzzy-valued function $\tilde{f}$ on $\mathbb{R}^n$ as follows:
\begin{definition}\label{def11}
We say that $\tilde{f}$ is gH-differentiable at $\bar{x}^{0}$ if and only if one of the partial gH-derivatives ${\partial\tilde{f}}/ {\partial x_{1}},..., {\partial\tilde{f}}/ {\partial x_{n}}$ exists at $\bar{x}^{0}$ and the remaining $n-1$ partial gH-derivatives exist on some neighborhood of $\bar{x}^{0}$ and are continuous at $\bar{x}^{0}$ (in the sense of fuzzy-valued function).
\end{definition}

The gradient of $\tilde{f}$ at $\bar{x}^{0}$ is denoted by 
\[
\nabla \tilde{f}(\bar{x}^{0}) = (D_{1}\tilde{f}(\bar{x}^{0}),..., D_{n}\tilde{f}(\bar{x}^{0})),
\]
 and it defines a fuzzy-valued function from $X$ to $F^{n}(\mathbb{R})=  F(\mathbb{R}) \times ....\times F(\mathbb{R})$ ($n$ times), where each $D_{i}\tilde{f}(\bar{x}^{0})$ is a fuzzy number for $i = 1,...,n$. We say that $\tilde{f}$ is gH-differentiable on $X$ if it is gH-differentiable at every $\bar{x}^{0} \in X$. 

\begin{definition}\label{def12}
We say that $\tilde{f}$ is continuously gH-differentiable at $\bar{x}^{0}$ if and only if all of the partial gH-derivatives ${\partial\tilde{f}}/ {\partial x_{i}}$, $i = 1,..,n$, exist on some neighborhood of $\bar{x}^{0}$ and are continuous at $\bar{x}^{0}$ (in the sense of fuzzy-valued function).\\
We say that $\tilde{f}$ is continuously gH-differentiable on $X$ if it is continuously gH-differentiable at every $\bar{x}^{0} \in X$. 
\end{definition}

\begin{definition}\label{def13}
Let $\tilde{f}:X \to F(\mathbb{R}), X \subset \mathbb{R}^{n}$ be a fuzzy-valued function. Suppose now that there is $\bar{x}^{0} \in X$ such that gradient of $\tilde{f}$, $\nabla \tilde{f}$,  is itself gH-differentiable at $\bar{x}^{0}$, that is, for each i, the function $D_{i}\tilde{f}: X \to F(\mathbb{R})$ is gH-differentiable at $\bar{x}^{0}$. Denote the gH-partial derivative of $D_{i}\tilde{f}$ in the direction of $\bar{e}_{j}$ at $\bar{x}^{0}$ by 
\[
D^{2}_{ij}\tilde{f}~~or~~\frac {\partial^{2}\tilde{f}(\bar{x}^{0})} {\partial x_{i} \partial x_{j}},~~if~~i \neq j,
\] and 
\[
D^{2}_{ii}\tilde{f}~~or~~\frac {\partial^{2}\tilde{f}(\bar{x}^{0})}{\partial x_{i}^{2}},~~if~~i = j.
\] 
Then we say that $\tilde{f}$ is twice gH-differentiable at $\bar{x}^{0}$, with second gH-derivative $\nabla ^{2} \tilde{f}(\bar{x}^{0})$ which is denoted by
\[ \nabla ^{2} \tilde{f}(\bar{x}^{0}) = \left(\begin{array}{ccc}
\frac{\partial ^{2}\tilde{f}(\bar{x}^{0})}{\partial x^{2}_{1}}& ...& \frac {\partial ^{2}\tilde{f}(\bar{x}^{0})} {\partial x_{1} \partial x_{n}} \\
... & ... &...\\
\frac {\partial ^{2}\tilde{f}(\bar{x}^{0})}{\partial x_{n} \partial x_{1}}& ...&\frac {\partial ^{2}\tilde{f}(\bar{x}^{0})}{\partial x^{2}_{n}} \\
\end{array}
\right)
\]
where $ \frac {\partial ^{2}\tilde{f}(\bar{x}^{0})} {\partial x_{i} \partial x_{j}} \in F(\mathbb{R})$, $i,j = 1,...,n$. \\
If $\tilde{f}$ is twice gH-differentiable at each $\bar{x}^{0}$ in $X$, we say that $\tilde{f}$ is twice gH-differentiable on $X$, and if for each $i, j = 1,...,n$, the cross-partial derivative $ \frac {\partial ^2 \tilde{f}}{\partial x_{i} \partial x_{j}}$ is continuous function from $X$ to $F(\mathbb{R})$, we say that $\tilde{f}$ is twice continuously gH-differentiable on $X$.
\end{definition}
\section{Optimal solution}

We consider an unconstrained multi-variable fuzzy optimization problem (UMFOP):
\[
Minimize~~~\tilde{f}(\bar{x}), ~~~~ \bar{x} \in X
\]
where $X \subseteq \mathbb{R}^{n}$ is an open set and $\tilde{f}: X \to F(\mathbb{R})$ is a fuzzy-valued function. \\

Using linear order relation defined in Section \ref{sec:2}, we can define optimum solution of (UMFOP) as follows:

\begin{definition} \label{defm}
Let $X \subseteq \mathbb{R}^{n}$ be an open set and Let $\tilde{f}: X \to F(\mathbb{R})$ be a fuzzy-valued function. 
\begin{itemize}
\item [1.] A point $\bar{x}^{*} \in X$ is said to be a local minimum or minimizer, if there exists a $\epsilon > 0$ such that $\tilde{f}(\bar{x}^{*})\preceq \tilde{f}(x)$, for all $\bar{x} \in N_{\epsilon}(\bar{x}^{*})$.
\item [2.] $\bar{x}^{*} \in X$ is said to be a global minimum or minimizer, if $\tilde{f}(\bar{x}^{*})\preceq \tilde{f}(\bar{x})$, for all $\bar{x} \in X$.
\item [3.] A point $\bar{x}^{*} \in X$ is said to be a strict local minimum or minimizer, if there exists a $\epsilon > 0$ such that $\tilde{f}(\bar{x}^{*})\prec \tilde{f}(\bar{x})$, for all $\bar{x} \in N_{\epsilon}(x^{*})$.
\item [4.] $\bar{x}^{*} \in X$ is said to be a strict global minimum or minimizer, if $\tilde{f}(\bar{x}^{*})\prec \tilde{f}(\bar{x})$, for all $\bar{x} \in X$.
\end{itemize}
\end{definition}
In the above, $\bar{x}^{*}$ is a local (global) maximizer if "$\preceq$" is replaced by "$\succeq$". 
To find an optimal solution of (UMFOP), we propose the first-order necessary condition: 

\begin{theorem} \label{thm1}
Suppose $\tilde{f}: X \to F(\mathbb{R})$ is gH-differentiable fuzzy-valued function, $X$ is an open subset of $\mathbb{R}^{n}$. If $\bar{x}^{*} \in X $ is a local minimizer of (UMFOP), then 
\[
\int_{0}^{1}\alpha [ \nabla ( \tilde{f}_{\alpha}^{L}+ \tilde{f}_{\alpha}^{U} )(\bar{x}^{*})] d\alpha  = 0.
\]
\end{theorem}
\begin{proof}
The proof is followed by hypothesis of the Theorem.
\end{proof}

\begin{definition}
$x^{*} \in X$ satisfying the property of the above theorem is called a critical point.
\end{definition}
The sufficient conditions show whether these points are minimizers/maximizers:

\begin{theorem}\label{thm2}
Let $\tilde{f}$ be a twice continuously gH-differentiable fuzzy-valued function defined on $X \subseteq \mathbb{R}^{n}$. If $\bar{x}^{*}$ is a local minimizer of (UMFOP) then 
\[
\int_{0}^{1}\alpha [ \nabla^{2} ( \tilde{f}_{\alpha}^{L}+ \tilde{f}_{\alpha}^{U} )(\bar{x}^{*})] d\alpha  
\]
is a positive semi-definite matrix.
\end{theorem}

Now, we propose a second-order sufficient condition.

\begin{theorem}\label{thm3}
Let $\tilde{f}$ be a twice continuously gH-differentiable function on $X \subseteq \mathbb{R}^{n}$. Suppose that
\begin{description}
\item [1.] 
\[
\int_{0}^{1}\alpha [ \nabla ( \tilde{f}_{\alpha}^{L}+ \tilde{f}_{\alpha}^{U} )(\bar{x}^{*})] d\alpha  = 0.
\]
\item [2.]
\[
\int_{0}^{1}\alpha [ \nabla^{2} ( \tilde{f}_{\alpha}^{L}+ \tilde{f}_{\alpha}^{U} )(\bar{x}^{*})] d\alpha  
\] is positive definite matrix.
\end{description}
Then, $\bar{x}^{*}$ is a strict local minimizer of (UMFOP).
\end{theorem}
\begin{remark}
The proofs of sufficient optimality conditions are straight forward and implemented in a same line as done by Pathak and Pirzada \cite{PA2}. 
\end{remark}
\section{Numerical applications}
First we consider an example for single-variable fuzzy optimization problem.
\begin{example}\label{example1}
\[
Minimize ~~~\tilde{f}(x) = (\tilde{a} \odot x^{3}) \oplus (\tilde{b} \odot x^{2}), ~x \in \mathbb{R} 
\] 
where $\tilde{a} = (a^L ,a,a^U ) $ and $\tilde{b} = (b^L ,b, b^U ) $ are appropriate triangular fuzzy numbers. \\

By fuzzy arithmetic, we have $\alpha$-level functions of the fuzzy-valued objective function are \\

\[ 
\tilde{f}_{\alpha}^{L}(x) = \left\{
\begin{array}{ll}
 ((1-\alpha) a^L + a \alpha) x^{3} + ((1-\alpha) b^L + b  \alpha) x^{2} ~~~ if ~~ x \geq 0   \\
 ((1-\alpha) a^U + a \alpha) x^{3} + ((1-\alpha) b^L + b  \alpha) x^{2} ~~~ if ~~ x < 0
\end{array} \right.
\] 
 and 
\[ 
\tilde{f}_{\alpha}^{U}(x) = \left\{
\begin{array}{ll}
 ((1-\alpha) a^U + a \alpha)  x^{3} + ((1-\alpha) b^U + b  \alpha) x^{2} ~~~ if ~~ x \geq 0   \\
 ((1-\alpha) a^L + a \alpha)  x^{3} + ((1-\alpha) b^U + b  \alpha) x^{2} ~~~ if ~~ x < 0
\end{array} \right.
\] 

Adding these two functions, we get 
\[(\tilde{f}_{\alpha}^{L} + \tilde{f}_{\alpha}^{U} )(x) = (((1-\alpha) a^L + a \alpha) + ((1-\alpha) a^U + a \alpha)x^{3} ) + (((1-\alpha) b^L + b  \alpha) + ((1-\alpha) b^U + b  \alpha)) x^{2} 
\]
which is three times continuously differentiable function. \\

Taking $\tilde{a} = \tilde{1} = (0,1,3)$ and $\tilde{b} = \widetilde{-12} = (-13, -12, -11)$ and applying the necessary condition, 

\begin{eqnarray*}
\int_{0}^{1} \alpha [ D (\tilde{f}_{\alpha}^{L} + \tilde{f}_{\alpha}^{U}) (x)] d\alpha = 0,
\end{eqnarray*}
we get $x = 0$ and  $x = 6.8571$ critical points. Now checking sufficient conditions at the critical points,

 \[
\int_{0}^{1} \alpha [ D^{2}( \tilde{f}_{\alpha}^{L} + \tilde{f}_{\alpha}^{U} ) (x)] d\alpha \geq 0
\]
 for all $x \in \mathbb{R}$, we observe that $x = 0$ is a strict local maximizer and $x = 6.8571$ is a strict local minimizer of given fuzzy-valued function.
\end{example}

\begin{remark}
If we change the spread of the fuzzy data in the given fuzzy objective function, for instance, if $\tilde{a} = \tilde{1} = (-1,1,3)$ and $\tilde{b} = \widetilde{-12} = (-14, -12, -11)$ then strict local minimizer is changed to $x = 8.111$ from $x = 6.8571$. If $\tilde{a} = \tilde{1} = (-1,1,2)$ and $\tilde{b} = \widetilde{-12} = (-13, -12, -10)$ then strict local minimizer is changed to $x = 9.4667$ from $x = 6.8571$. If $\tilde{a} = \tilde{1} = (0,1,2)$ and $\tilde{b} = \widetilde{-12} = (-13, -12, -11)$ then strict local minimizer is changed to $x = 8$ from $x = 6.8571$ which is the same solution of the crisp objective function.
\end{remark}

\begin{example}\label{example2}
\[
Minimize ~~~\tilde{f}(x) = (\tilde{a} \odot x^{2}) \oplus (\tilde{b} \odot (-x)) \oplus \tilde{v}, ~x \in \Omega \subseteq \mathbb{R} 
\] 
where $\Omega $ is the set of all positive real numbers. $\tilde{a} = (a^L ,a,a^U ) $, $\tilde{b} = (b^L ,b, b^U ) $ and $\tilde{v} = (v^L ,v,v^U ) $ are appropriate triangular fuzzy numbers. \\

By fuzzy arithmetic, we have $\alpha$-level functions of the fuzzy-valued objective function are \\

\begin{equation}
\tilde{f}_{\alpha}^{L}(x) = 
 ((1-\alpha) a^L + a \alpha) x^{2} - ((1-\alpha) b^U + b  \alpha) x + ((1-\alpha) v^L + v  \alpha) 
\end{equation} 
 and 
\begin{equation}
\tilde{f}_{\alpha}^{U}(x) = 
 ((1-\alpha) a^U + a \alpha) x^{2} - ((1-\alpha) b^L + b  \alpha) x + ((1-\alpha) v^U + v  \alpha) 
\end{equation} 

Adding these two functions, we get 

\begin{eqnarray*}
(\tilde{f}_{\alpha}^{L} + \tilde{f}_{\alpha}^{U} )(x) 
& = & (((1-\alpha) a^L + a \alpha) + ((1-\alpha) a^U + a \alpha)x^{2} ) + (((1-\alpha) b^L \\
& + & b  \alpha) + ((1-\alpha) b^U + b  \alpha)) x +  ((1-\alpha) v^L + v \alpha) + ((1-\alpha) v^U + v \alpha)
\end{eqnarray*}
which is three times continuously differentiable function. \\

Taking $\tilde{a} = (0,1,4)$, $\tilde{b} = (0,3,4)$ and $\tilde{v} = (1,2,4)$. Applying the necessary condition, 

\begin{eqnarray*}
\int_{0}^{1} \alpha [ D (\tilde{f}_{\alpha}^{L} + \tilde{f}_{\alpha}^{U}) (x)] d\alpha = 0,
\end{eqnarray*}
we get a critical point $x = 1$. Now checking sufficient conditions at the critical point,

 \[
\int_{0}^{1} \alpha [ D^{2}( \tilde{f}_{\alpha}^{L} + \tilde{f}_{\alpha}^{U} ) (x)] d\alpha \geq 0
\]
 for all $x \in \Omega $, we observe that $x = 1$ is a strict local minimizer of given fuzzy-valued function.
\end{example}

\begin{remark}
It has been observed that the fuzzy-valued function $\tilde{f}(x) = ((0,1,4) \odot x^{2}) \oplus ((0,3,4) \odot (-x)) \oplus (1,2,4)$, $x \in \Omega $ is the set of all positive real numbers, is not comparable under the partial order relation- fuzzy-max order (see \cite{PA1}). But the fuzzy-valued function is comparable under linear order stated in this paper. By considering the fuzzy optimization problem under linear order relation, we have obtained a local minimizer point.
\end{remark}

Another illustration is for multi-variable fuzzy optimization problem taken from \cite{PA}.
\begin{example}\label{ex1}
Let $\tilde{f}: \mathbb{R}^{2} \to F(\mathbb{R})$ be defined by $\tilde{f}(x_{1},x_{2}) = (0,2,4)\odot x_{1}^{2} \oplus (0,2,4)\odot x_{2}^{2} \oplus (1,3,5)$, where $(0,2,4)$ and $(1,3,5)$ are triangular fuzzy numbers. \\

This fuzzy-valued function is not H-differentiable as in the first (second) term, the derivative of the real-valued function  
$x_{1}^{2}$ ($x_{2}^{2}$) is negative. Therefore, we can not apply necessary and sufficient optimality conditions 
proposed Pathak and Pirzada \cite{PA} for this function. But if we restrict the domain of fuzzy-valued function $\mathbb{R}^{2}_{+}$ then 
we can see that the function is H-differentiable. By applying first order necessary condition of Pathak and Pirzada \cite{PA}, the critical point 
is $\bar{x}^{*} = (0,0)$ but this point is not a interior point of the domain so we can not check the sufficient conditions. 

It is to be noticed that the fuzzy-valued function in \textit{Example \ref{ex1}} is not H-differentiable but it is gH-differentiable on $\mathbb{R}$. By applying first order necessary condition proposed in this paper, we get the critical point $\bar{x}^{*} = (0,0)$. We can easily verify the proposed results of second order necessary and sufficient conditions for the minimization of the critical point $\bar{x}^{*} = (0,0)$. Therefore, we say that $\bar{x}^{*} = (0,0)$ is a strict local minimum of the given fuzzy-valued function under linear order relation. 
\end{example}
We discuss a case study as an illustration. 

\begin{example} The decision maker allows to vary the coefficients of the objective function ( the approximate profit ) in a certain range rather than exact real numbers. This variation in the real numbers is represented mathematically using fuzzy numbers. Therefore, in the situation when the coefficients are known approximately, the approximate profit per acre of a certain farm is represented by the following non-linear fuzzy-valued function:
\[
\tilde{f}(x_{1},x_{2}) = \widetilde{20}\odot x_{1} \oplus \widetilde{26} \odot x_{2} \oplus (\tilde{4} \odot x_{1}\cdot x_{2})) \oplus (\widetilde{-4})\odot x_{1}^{2}) \oplus (\widetilde{-3})\odot x_{2}^{2}), ~(x_{1},x_{2}) \in \Omega \subset\mathbb{R}^2, 
\] 
where $\widetilde{20}=(19, 20, 21)$, $\widetilde{26}=(25, 26, 27)$, $\tilde{4}=(2, 4, 6)$, $\widetilde{-4}=(-6, -4, -2)$, $\widetilde{-3}=(-5, -3, -2)$ are fuzzy numbers and $x_{1}, x_{2}$ denotes respectively, the labour cost and the fertilizer cost. We have to determine the value of $x_{1}$ and  $x_{2}$ to maximize the approximate profit. \\

Using fuzzy arithmetic, we can write
\[ 
\tilde{f}_{\alpha}^{L}(x_{1},x_{2}) = \left\{
\begin{array}{ll}
 (19 + \alpha)x_{1}+ (25 + \alpha)x_{2}+(2+2\alpha)x_{1}\cdot x_{2}+(-6 + 2\alpha)x_{1}^2 +(-5+2\alpha)x_{2}^2  , ~~~ if ~~ x_{1},x_{2} \geq 0   \\
(21 - \alpha)x_{1}+ (25 + \alpha)x_{2}+(6-2\alpha)x_{1}\cdot x_{2}+(-6 + 2\alpha)x_{1}^2 +(-5+2\alpha)x_{2}^2  , ~~~ if ~~ x_{1} <0,x_{2} \geq 0   \\
(19 + \alpha)x_{1}+ (27 - \alpha)x_{2}+(6-2\alpha)x_{1}\cdot x_{2}+(-6 + 2\alpha)x_{1}^2 +(-5+2\alpha)x_{2}^2  , ~~~ if ~~ x_{1}\geq 0,x_{2} <0   \\
 (21 - \alpha)x_{1}+ (27 - \alpha)x_{2}+(2+2\alpha)x_{1}\cdot x_{2}+(-6 + 2\alpha)x_{1}^2 +(-5+2\alpha)x_{2}^2 ~~~ if ~~ x_{1},x_{2} < 0
\end{array} \right.
\] 
 and 
\[ 
\tilde{f}_{\alpha}^{U}(x_{1},x_{2}) = \left\{
\begin{array}{ll}
 (21 - \alpha)x_{1}+ (27 - \alpha)x_{2}+(6-2\alpha)x_{1}\cdot x_{2}+(-2 - 2\alpha)x_{1}^2 +(-2-\alpha)x_{2}^2 , ~~~ if ~~ x_{1},x_{2} \geq 0   \\
(19 + \alpha)x_{1}+ (27 - \alpha)x_{2}+(2+2\alpha)x_{1}\cdot x_{2}+(-2 - 2\alpha)x_{1}^2 +(-2-\alpha)x_{2}^2  , ~~~ if ~~ x_{1} <0,x_{2} \geq 0   \\
(21 - \alpha)x_{1}+ (25 + \alpha)x_{2}+(2+2\alpha)x_{1}\cdot x_{2}+(-2 - 2\alpha)x_{1}^2 +(-2-\alpha)x_{2}^2  , ~~~ if ~~ x_{1}\geq 0,x_{2} <0   \\
 (21 - \alpha)x_{1}+ (27 - \alpha)x_{2}+(6-2\alpha)x_{1}\cdot x_{2}+(-2 - 2\alpha)x_{1}^2 +(-2-\alpha)x_{2}^2 ~~~ if ~~ x_{1},x_{2} < 0.
\end{array} \right.
\] 
Though the level functions $\tilde{f}_{\alpha}^{L}(x_{1},x_{2})$ and $\tilde{f}_{\alpha}^{U}(x_{1},x_{2})$ are not differentiable at $x_{1} = x_{2} = 0 $, the fuzzy-valued function is gH-differentiable. By adding the level functions, we get 

\begin{eqnarray*}
(\tilde{f}_{\alpha}^{L} + \tilde{f}_{\alpha}^{U} )(x) 
= 40x_{1} + 52x_{2} + 8x_{1}x_{2} - 8x_{1}^{2} + (-7 + \alpha)x_{2}^{2}.
\end{eqnarray*}
We apply the necessary optimality condition, 

\begin{eqnarray*}
\int_{0}^{1} \alpha [ \nabla (\tilde{f}_{\alpha}^{L} + \tilde{f}_{\alpha}^{U}) (x)] d\alpha = 0.
\end{eqnarray*}
Which gives a critical point $(x_{1},x_{2}) = (173/26,108/13)$. Checking sufficient condition at the critical point, i.e.

 \[
\int_{0}^{1} \alpha [ \nabla^{2}( \tilde{f}_{\alpha}^{L} + \tilde{f}_{\alpha}^{U} ) (x)] d\alpha 
\]
 negative definite, for all $x \in \Omega $. Therefore, the critical point is a strict local maximizer of given fuzzy-valued function.

\end{example}
 

\section{Conclusions}
The local and global minimizer of multi-variable fuzzy optimization problem are studied using the necessary and sufficient 
optimality conditions. The optimality conditions to obtain the solution of unconstrained fuzzy optimization problems are 
proposed earlier Pathak and Pirzada \cite{PA},\cite{PA2}. The comparison of the optimality conditions is given in \textbf{Table 1}.
The class of Hukuhara and generalized Hukuhara differentiable functions are given in \textbf{Table 2}.
\begin{center}
\begin{table}
\caption{Comparison of optimality conditions}
\label{tab:1}
     \begin{tabular}{ | p{2.0cm} | p{2.5cm} | p{2.5cm} | p{4.5cm} |}

\hline
Research work & Hypothesis & Optimality conditions & Comments \\ \hline
Pathak and Pirzada (2011) \cite{PA2} & {Hukuhara differentiability and Linear order on $F_L(\mathbb{R})$} & Parametric optimality conditions in terms of lower level functions $\tilde{f}_{\alpha}^{L}$ with $\alpha = 0,1$ &  
\begin{itemize}
\item[1.] Hukuhara differentiability exists under very restrictive conditions.  
\item[2.] Limited class of fuzzy-valued functions are H-differentiable. 
\item[3.] The order relation is linear, but it is defined on L shaped  fuzzy numbers. 
\item[4.] The optimality conditions obtained are in parametric form.
\end{itemize}  \\ \hline
Pathak and Pirzada (2013) \cite{PA} & Hukuhara differentiability and Partial order on $F(\mathbb{R})$ &  Optimality conditions in terms of $\alpha$'s & 
\begin{itemize}
\item[1.] Hukuhara differentiability exists under very restrictive conditions. 
\item[2.] Limited class of fuzzy-valued functions are H-differentiable.
\item[3.] The order relation is partial, so not all fuzzy numbers are comparable under this order relation. 
\end{itemize} \\ \hline
Pirzada Current work & Generalized Hukuhara differentiability and Linear order on $F(\mathbb{R})$ & Optimality conditions are in integral form & 
\begin{itemize}
\item[1.] Generalized Hukuhara differentiability is more general. 
\item[2.] More class of functions are gH-differentiable (See \textbf{Table 2}). 
\item[3.] The linear order relation is defined on whole class of fuzzy numbers. 
\item[4.] Any two fuzzy numbers are comparable under this order relation. 
\item[5.] The optimality conditions are in more general form as they are independent of any parameters and $\alpha$'s. 
\item[6.] We can apply these conditions to bigger (general) class fuzzy optimization problems.  
\end{itemize} \\ \hline

    \end{tabular}
\end{table}
\end{center}

\begin{center}
\begin{table}
\caption{Class of Hukuhara and generalized Hukuhara differentiable functions (Pirzada and Vakaskar \cite{PI1})}
\label{tab:2}
     \begin{tabular}{ | p{4.0cm} | p{3.5cm} | p{4.5cm} |}
\hline
Fuzzy-valued function & Hukuhara differentiability & Generalized Hukuhara differentiability \\ \hline
Constant function: $\tilde{f}(x) = \tilde{a}$, $\tilde{a}$ is a fuzzy number & For all $x$ in $\mathbb{R}$ & For all $x \in \mathbb{R}$ \\ \hline
Linear function: $\tilde{f}(x) = \tilde{a}\odot x$ & For $x > 0$ & For all $x \in \mathbb{R}$  \\ \hline
Quadratic function: $\tilde{f}(x) = \tilde{a} \odot x^2$ & For $x > 0$ & For all $x \in \mathbb{R}$ \\ \hline
$\tilde{f}(x) = \tilde{a}\odot x^{n}$, $n$ is a positive integer & n-times H-differentiable for $x \in \mathbb{R}_{+}$ & n-times gH-differentiable for all $x \in \mathbb{R}$  \\ \hline
$\tilde{f}(x) = \tilde{a}_{n}\odot x^{n} \oplus \tilde{a}_{n−1} \odot x^{n−1} \oplus ... \oplus \tilde{a}_{1} \odot x \oplus \tilde{a}_{0}
$, $n \geq 0$ & n-times H-differentiable for $x \in \mathbb{R}_{+}$ & n-times H-differentiable for $x \in \mathbb{R}$  \\ \hline
Sine function: $\tilde{f}(x) = \tilde{a} \odot sin(x)$ defined on $[\pi, \pi]$ & One time H-differentiable in $[0, \pi/2]$ & n-times gH-differentiable in $[0,\pi]$  \\ \hline
Exponential function: $\tilde{f}(x) = \tilde{a}\odot e^{x} $ & n-times differentiable on $\mathbb{R}_{+} \cap \{0\}$ & n-times gH-differentiable on $\mathbb{R}$  \\ \hline
    \end{tabular}
\end{table}
\end{center}
\begin{acknowledgements}
UMP is thankful to National Board for Higher Mathematics (NBHM), Department of Atomic Energy (DAE), India, for supporting this 
work. UMP is also thankful to B S Ratanpal (Department of Applied Mathematics, The Maharaja Sayajirao University of Baroda, India)
for necessary suggestions.
\end{acknowledgements}

\end{document}